\documentclass{article}

\usepackage{amsfonts}
\usepackage{amsmath}
\usepackage{theorem}
\usepackage{array}
\usepackage[a4paper]{geometry}
\usepackage{amssymb}
\usepackage{graphics}
\usepackage[usenames]{color}
\usepackage[all]{xy}
\usepackage{graphicx}
\usepackage{boxedminipage}

\newtheorem{theorem}{Theorem}[section]

\newtheorem{remark}[theorem]{Remark}

\newtheorem{lemma}[theorem]{Lemma}

\newenvironment{proof}[1][Proof]{\textbf{#1.} }{\ \rule{0.5em}{0.5em}}

\newcommand{\f}{\mathbb{F}}

\newcommand{\fq}{\mathbb{F}_q}

\newcommand{\fqc}{\mathbb{F}_{q^3}}

\begin{document}
\title{A modular interpretation of various cubic towers}
\author{Nurdag\"ul Anbar, Alp Bassa and Peter Beelen}
\date{\empty}

\maketitle

\begin{abstract} In this article we give a Drinfeld modular interpretation for various towers of function fields meeting Zink's bound.
\end{abstract}



\section{Introduction}

Let $p$ be a prime number and $q=p^n$ for a positive integer $n$. It is a central question in algebraic geometry how many $\mathbb F_{q}$-rational places $N(F)$ a function field $F$ with full constant field $\mathbb F_{q}$ and genus $g(F)$ can have. For small genus the Hasse--Weil estimate $N(F) \le q+1+2g(F) q^{1/2}$ is good, but the larger the genus compared to the size of the finite field $\mathbb F_{q}$ it gets worse. For this reason, Ihara introduced the constant $$A(q):=\limsup_{g(F) \to \infty} \frac{N(F)}{g(F)} \ ,$$ where the limit is taken over all function fields $F$ with full constant field $\mathbb F_{q}$ and genus tending to infinity.
It is known that $0< A(q) \le \sqrt{q}-1$, the first inequality being due to Serre \cite{serre}, while the second inequality is known as the Drinfeld--Vladut bound \cite{Vladut1983}. Combining the work of Ihara \cite{Ihara1982} and the Drinfeld--Vladut bound, one sees that $A(q)=\sqrt{q}-1$ if $q$ is square; i.e. $n$ is even.

Ihara used reductions of modular curves to obtain his result. It was a surprise when in \cite{GS_INV} completely different methods were used to prove the same result. In \cite{GS_INV} a lower bound for $A(q)$ is obtained using a \emph{tower (of function fields) over $\fq$}, $$\mathcal F=(F_0 \subseteq F_1 \subseteq F_2\subseteq \ldots \subseteq F_i\subseteq \ldots) \ .$$ It is required that all function fields $F_i$ have full constant field $\fq$, and $g(F_i) \to \infty$ as $i \to \infty$. Also all extensions $F_{i+1}/F_i$ are assumed to be separable. These assumptions imply that the following limit exists:
$$\lambda(\mathcal F):=\lim_{i\to \infty} \frac{N(F_i)}{g(F_i)} \ ,$$
which is called the limit of the tower $\mathcal F$.
One then obtains the lower bound for Ihara's constant: $A(q) \ge \lambda(\mathcal F).$
The main ingredient in \cite{GS_INV} was to explicitly produce a tower of function fields over $\mathbb{F}_{q}$ with limit $q^{1/2}$. This method also turned out to be fruitful for nonsquare $q$.
For $p=2$, using explicit towers of function fields, it was shown in \cite{vdGvdV2002} that $A(8) \ge 3/2$, while this result was generalized in \cite{BGS2005} to
\begin{equation}\label{eq:zink}
A(q^{3}) \ge \dfrac{2(q^{2}-1)}{q+2} \ .
\end{equation}
The generalization was achieved by explicitly constructing a tower of function fields over $\fqc$ (which we will call a \emph{cubic} tower) with limit $\lambda(\mathcal F) \ge 2(q^2-1)/(q+2).$ Since then several other papers have appeared in which other towers or alternative descriptions of previously known towers were formulated, giving rise to various cubic towers with the same limit \cite{I2007,CG2012,BaBeGS2013}.
Note that these results generalize the statement by Zink in \cite{zink} that $A(p^3) \ge 2(p^2-1)/(p+2)$ for a prime $p$. For this reason the bound in Equation \eqref{eq:zink} is called Zink's bound.

The different methods of using explicit towers on the one hand and reductions of modular curves on the other hand became interlinked when Elkies showed that the towers in \cite{GS_INV} can also be obtained using the theory of Drinfeld modules in \cite{elkies}. Despite the recent developments in \cite{BaBeGS2013} and \cite{Gek}, where the theory of Drinfeld modules and their moduli spaces was used to construct sequences of curves with many rational points over any non-prime field, it remained a mystery to what extent the original cubic towers meeting Zink's bound, can be explained from the modular theory. In this article we solve this problem.

\section{Modular setup and first equations}

We start by giving a brief introduction to Drinfeld modules, since these will be needed in the remainder of the paper. See \cite{Goss1996} for a more thorough and general overview.

\subsection{Drinfeld modules over $\fq[T]$}

Let $L$ be a field and $\overline{L}$ a fixed algebraic closure. Moreover, assume that $\iota: \fq[T] \rightarrow L$ is a $\fq$-algebra homomorphism. The kernel of $\iota$ is called the \emph{$\mathbb{F}_q[T]$-characteristic} of $L$. From now on, we will always assume that $L$ has $\mathbb{F}_q[T]$-characteristic $\langle T-1 \rangle \subset \fq[T]$ and by slight abuse of language also call the polynomial $T-1$ the $\mathbb{F}_q[T]$-characteristic of $L$. Note that this assumption implies that for any $P(T) \in \fq[T]$ we have $\iota(P(T))=P(1)$, the evaluation of the polynomial $P(T)$ at $1$. Now let $L\{\tau\}$ be the \emph{non-commutative polynomial ring} generated by the Frobenius endomorphism $\tau$ satisfying $\tau r = r^q \tau$ for all $r \in L$. Then an \emph{$\fq[T]$-Drinfeld module} over $L$ of \emph{rank} $3$ is a homomorphism
\begin{align*}
\varphi : \fq[T] &\to L\{\tau\}\\
P(T) &\mapsto \varphi_{P(T)}
\end{align*}
such that for all $P(T) \in \fq[T]\backslash\{0\}$, we have $\deg_\tau \varphi_{P(T)} = 3 \deg P(T)$, and the constant term of $\varphi_{P(T)}$ is equal to $\iota(P(T))$, i.e., equal to $P(1)$. This gives $\bar{L}$ the structure of an $\mathbb F_q[T]$-module. Since $\varphi$ is a homomorphism, it is already fully determined by $\varphi_T$. Therefore by slight above of language, we will talk about the Drinfeld module $\varphi_T$. Using that the rank of $\varphi$ is $3$, we see that $$\varphi_T=\Delta\tau^3+g\tau^2+h\tau+1 \ ,$$ for certain $\Delta,g,h \in L$ and $\Delta \neq 0$.

Two Drinfeld modules $\varphi$ and $\psi$ with the same $\mathbb{F}_q[T]$-characteristic are called \emph{isogenous} if there exists $\lambda \in L\{\tau\}$ different from zero such that
\begin{equation}
\lambda \circ \varphi_{P(T)}=\psi_{P(T)} \circ \lambda
\label{isogeny}
\end{equation}
for all $P(T)\in \mathbb F_q[T]$. The element $\lambda$ is called an \emph{isogeny} from $\varphi$ to $\psi$. Since we are considering $\fq[T]$-Drinfeld modules, it is sufficient to require that $\lambda \circ \varphi_{T}=\psi_{T}\circ \lambda$ for $\lambda$ to be an isogeny. It is easy to see that $\varphi_{P(T)}$ is an isogeny from $\varphi$ to itself for any $P(T)\in \fq[T]$. This isogeny is called the \emph{multiplication by $P(T)$ map}. The Drinfeld modules $\varphi$ and $\psi$ are called \emph{isomorphic (over $\overline{L})$} if $\lambda$ can be chosen from $\overline{L} \backslash\{0\}$.

An isogeny $\lambda \in L\{\tau\}$ corresponds to a linearized polynomial by identifying $\tau^i$ and $X^{q^i}$. This makes it possible to evaluate $\lambda$ at elements of $\overline{L}$. In particular we define the kernel of an isogeny $\lambda$ as follows:
$$\ker \lambda:=\{x\in \bar{L}|\lambda(x)=0\} \ .$$
From Equation~\eqref{isogeny} it follows that $\ker \lambda=\{x\in \bar{L}|\lambda(x)=0\}$ is an $\fq[T]$-submodule of $\bar{L}$ under the $\fq[T]$-action given by $\varphi$. For $P(T) \in \fq[T]$, we write $\varphi[P(T)]:=\ker \varphi_{P(T)}.$ This set is called the set of \emph{$P(T)$-torsion points} of the Drinfeld module $\varphi$. If $P(T)$ is coprime with the $\mathbb{F}_q[T]$-characteristic $T-1$, then $\varphi[P(T)] \cong(\fq[T]/\langle P(T) \rangle)^{3}$ as an $\fq[T]$-module, i.e., it is a free $\mathbb F_q[T]/\langle P(T) \rangle$-module of rank $3$.

An isogeny $\lambda$ is called a \emph{$P(T)$-isogeny} if $\ker \lambda$ is a free $\fq[T]/\langle P(T) \rangle$-submodule of $\varphi[P(T)]$. The \emph{rank of $\lambda$} is defined to be the rank of its kernel as a $\fq[T]/\langle P(T) \rangle$-module. Unlike in the classical case of elliptic curves, nontrivial isogenies can have rank $1$ or $2$.
For a separable $P(T)$-isogeny $\lambda$ one can find $\mu \in L\{\tau\}$ such that $\varphi_{P(T)}=\mu \circ \lambda$.

In case $P(T)=T-1$, we have $\varphi_{T-1}=\Delta\tau^3+g\tau^2+h\tau$ and therefore $\varphi[T-1]$ is isomorphic to a free $\fq[T]/\langle T-1 \rangle$-module of rank at most $2$. Classically, the Drinfeld module is called \emph{supersingular} (in $\mathbb{F}_q[T]$-characteristic $T-1$), if the rank of $\varphi[T-1]$ is zero, i.e., if the multiplication by $T-1$ map $\varphi_{T-1}$ is purely inseparable. This is the case if and only if $g=h=0$. We will call a Drinfeld module $\varphi$ \emph{weakly supersingular} (in $\mathbb{F}_q[T]$-characteristic $T-1$), if $(T-1)$-torsion points $\varphi[T-1]$ form a free module of rank at most one. In this case the multiplication by $T-1$ map has inseparability degree $\ge q^2$ and $h=0$. Comparing the inseparability degrees of $\phi_{T-1}$ and $\psi_{T-1}$ by Equation~\eqref{isogeny}, we see that the property of being (weakly) supersingular is preserved under isogenies and in particular under isomorphisms. From now on we will restrict ourselves to weakly supersingular Drinfeld modules and their isogenies.

For a weakly supersingular $\fq[T]$-Drinfeld module of rank $3$ given by $\varphi_T=\Delta\tau^3+g\tau^2+1$, we define the following $J$-invariant:
\begin{equation*}
J(\varphi):=\dfrac{g^{q^2+q+1}}{\Delta^{q+1}} \ .
\end{equation*}
Two Drinfeld modules $\varphi_T=\Delta_1\tau^3+g_1\tau^2+1$ and $\psi_T=\Delta_2\tau^3+g_2\tau^2+1$ are isomorphic if and only if $J(\varphi)=J(\psi).$ Indeed, if $c \varphi = \psi c$ for some nonzero constant $c \in \overline{L}$, then $c^{q^3-1}\Delta_2=\Delta_1$ and $c^{q^2-1}g_2=g_1$, implying that $J(\varphi)=J(\psi).$ Conversely, if $J(\varphi)=J(\psi)$, then $(\Delta_1/\Delta_2)^{q+1}=(g_1/g_2)^{q^2+q+1}$. We can find $c \in \overline{L}$ such that $c^{q^3-1}=\Delta_1/\Delta_2$ and from the previous we see that $c^{(q^2-1)(q^2+q+1)}=(g_1/g_2)^{q^2+q+1}.$ Therefore we can choose $\alpha$ satisfying $\alpha^{q^2+q+1}=1$ such that $c':=\alpha c$ satisfies both $(c')^{q^3-1}=\Delta_1/\Delta_2$ and $(c')^{q^2-1}=g_1/g_2$. The desired isomorphism between $\varphi$ and $\psi$ is then given by $c'$. Note that the supersingular Drinfeld modules of rank $3$ form one isomorphism class determined by $J(\varphi)=0$. In fact any supersingular Drinfeld module $\varphi$ and its isogenies can be defined over $\f_{q^3}$ (see \cite{G1991}).

\subsection{Normalized Drinfeld modules}
Since the expression for the $J$-invariant is somewhat cumbersome, rather than working with isomorphism classes directly, we will use normalized Drinfeld modules. An $\fq[T]$-Drinfeld module of rank $3$ is said to be normalized if $\Delta=-1$. This is a direct generalization of a similar notion used in \cite{elkies} for rank $2$ Drinfeld modules. Any isomorphism class of Drinfeld modules contains a normalized one, but two distinct normalized Drinfeld modules can be isomorphic. The $J$-invariant of the normalized weakly supersingular Drinfeld module $\varphi_T=-\tau^3+g_1\tau^2+1$ is given by:
\begin{equation}\label{eq:jinvariant}
J(\varphi)=g_1^{q^2+q+1} \ .
\end{equation}
Since we will be working with normalized Drinfeld modules or isomorphism classes thereof, we will take all isogenies to be monic.

Now let $\lambda:\varphi \to \psi$ be a separable monic $T$-isogeny of rank $1$. Then $\lambda=\tau-u_1$ with
$$\lambda \circ \varphi_T=\psi_T \circ \lambda$$
and there exists $\mu\in L\{\tau\}$ such that
$$\varphi_T=\mu\circ \lambda \ . $$
These imply that $\psi_T=\lambda \circ \mu$.
From
$$\mu \circ \lambda=\varphi_T=-\tau^3+g_1\tau^2 +1$$
it follows that
\begin{equation}\label{eq:mu and g}
\mu=-\tau^2-\frac{1}{u_1^{q+1}}\tau-\frac{1}{u_1}\ , \ {\rm with}\ g_1=\dfrac{u_1^{q^2+q+1}-1}{u_1^{q+1}}
\end{equation}
and
\begin{equation}\label{eq:h}\psi_T=-\tau^3+g_2\tau^2+1 \ ,  \ {\rm with}\ g_2=\dfrac{u_1^{q^2+q+1}-1}{u_1^{q^2+q}} \ . \end{equation}

The following lemma follows which will be useful later:
\begin{lemma}\label{lem:ghu}
With the relations as above, we have $$\fq(g_1,g_2)=\fq(u_1)\ \makebox{and} \ \fq(g_1^{q^2+q+1},g_2^{q^2+q+1})=\fq(u_1^{q^2+q+1}) \ . $$
\end{lemma}
\begin{proof}
First observe that in the extension $\fq(u_1)/\fq(g_1/g_2)$ only tame ramification occurs, since $g_1/g_2=u_1^{q^2-1}$. In particular, the same holds for the extension $\fq(u_1)/\fq(g_1,g_2)$. On the other hand there exists a place of $\fq(u_1)$ lying above the pole of $g_1$ in $\fq(g_1)$ which has ramification index $q^2$. Since all ramification in $\fq(u_1)/\fq(g_1,g_2)$ is tame, we conclude that the extension degree $[\fq(g_1,g_2):\fq(g_1)]$ is at least $q^2$. However, $[\fq(g_1,g_2):\fq(g_1)]$ also divides $[\fq(u_1):\fq(g_1)]=q^2+q+1$, implying that $\fq(g_1,g_2)=\fq(u_1)$ as desired. The second part of the lemma can be shown similarly considering the $(q^2+q+1)$-st powers of all variables involved.
\end{proof}

By Equation \eqref{eq:jinvariant} the quantity $g_1^{q^2+q+1}$ (resp. $g_2^{q^2+q+1}$) is the $J$-invariant of the normalized supersingular Drinfeld module $\varphi_T=-\tau^3+g_1\tau^2+1$ (resp. $\psi_T=-\tau^3+g_2\tau^2+1$). The above lemma is useful, since it implies that if two such Drinfeld modules are isogenous by an isogeny $\lambda=\tau-u_1$, then $u_1^{q^2+q+1}$ can be used as a parameter to describe isomorphism classes of the data $\lambda: \varphi \to \psi$ with $\lambda=\tau-u_1$.

\section{Composition of $T$-isogenies}

Next we will study composites of $T$-isogenies between normalized Drinfeld modules of rank $3$. The various possibilities of their structure will later on be the main ingredient in our explanation of the cubic towers in \cite{BGS2005,I2007,CG2012} from a modular point of view.

\subsection{Composition of two $T$-isogenies}

Let $\varphi^{(1)}, \varphi^{(2)}, \varphi^{(3)}$ be Drinfeld modules with $\varphi_T^{(i)}=-\tau^3+g_i\tau^2+1$, and $\lambda_1: \varphi^{(1)}\to\varphi^{(2)}$ and $\lambda_2:\varphi^{(2)}\to\varphi^{(3)}$ be $T$-isogenies of rank $1$ with $\lambda_i=\tau-u_i$. Since $\lambda_1$ and $\lambda_2$ are separable $T$-isogenies, there exist $\mu_1$ and $\mu_2$ such that $\varphi_T^{(1)}=\mu_1\circ \lambda_1$ and $\varphi_T^{(2)}=\lambda_1\circ \mu_1=\mu_2\circ \lambda_2$. We wish to find an algebraic relation between $u_1$ and $u_2$. Combining Equation~\eqref{eq:mu and g} applied to $\varphi^{(2)}$ and $\lambda_2$ with Equation~\eqref{eq:h} applied to $\varphi^{(1)}$ and $\lambda_1$, we see that
\begin{equation}\label{eq:u1u2}\dfrac{u_2^{q^2+q+1}-1}{u_2^{q+1}}=g_2=\dfrac{u_1^{q^2+q+1}-1}{u_1^{q^2+q}}.\end{equation}
Clearing denominators we have
\begin{multline*}
0=(u_2^{q^2+q+1}-1)u_1^{q^2+q}-u_2^{q+1}(u_1^{q^2+q+1}-1)\\= \left(u_1^{q+1}u_2^{q+1}+u_2+u_1^q\right)\left(u_2\cdot \Bigl(u_1^{q+1}u_2^{q+1}+u_2+u_1^q \Bigr)^{q-1}-u_1^{q^2}\right) \ .
\end{multline*}
We will now recover these two factors in Equations~\eqref{eq:factor1} and \eqref{eq:factor2} using the modular theory, thus giving them a modular interpretation.
The composite $\lambda_{2} \circ \lambda_1$ will be an isogeny from $\varphi^{(1)}$ to $\varphi^{(3)}$, with
$$\ker \lambda_2 \circ \lambda_1 \subseteq \varphi^{(1)}[T^2] \ .$$
Since $\lambda_2\circ \lambda_1$ defines a separable map of degree $q^2$, under the $\phi^{(1)}$-action $\ker \lambda_2 \circ \lambda_1$ is an $\mathbb F_q[T]/\langle T^2\rangle$-module with $q^2$ elements. Hence it will be isomorphic to $\langle T\rangle/\langle T^2\rangle\bigoplus\langle T\rangle/\langle T^2\rangle$ or $\mathbb F_q[T]/\langle T^2\rangle$ as an $\mathbb F_q[T]/\langle T^2\rangle$-module.  In the first case $\ker \lambda_2 \circ \lambda_1$  is annihilated by $\varphi^{(1)}_T$, so $\lambda_2 \circ \lambda_1$ is a $T$-isogeny of $\varphi^{(1)}$ of rank $2$. In the second case $\ker \lambda_2 \circ \lambda_1$ is a free $\mathbb F_q[T]/\langle T^2\rangle$-submodule of $\varphi^{(1)}[T^2]$, so $\lambda_2 \circ \lambda_1$ is a $T^2$-isogeny of rank $1$. We will consider these two cases separately in detail.
\begin{itemize}
\item In the first case where $\lambda_2 \circ \lambda_1$ is a right factor of $\varphi^{(1)}_T=\mu_1 \circ \lambda_1$, we have that $\lambda_2$ is a right factor of $\mu_1$ already. 
    Let $x_2$ be a nonzero $T$-torsion point in the kernel of $\lambda_2$, i.e., $x_2^{q-1}=u_2$. Then, using Equation~\eqref{eq:mu and g}:
$$\mu_1(x_2)=-x_2^{q^2}-\frac{1}{u_1^{q+1}}x_2^q-\frac{1}{u_1}x_2=0 \ . $$
Dividing by $x_2$ we have
$$-(x_2^{q-1})^{q+1}-\frac{1}{u_1^{q+1}}x_2^{q-1}-\frac{1}{u_1}=-u_2^{q+1}-\frac{1}{u_1^{q+1}}u_2-\frac{1}{u_1}=0 \ . $$ After clearing denominators, we obtain \begin{equation}\label{eq:factor1}u_1^{q+1}u_2^{q+1}+u_2+u_1^q=0 \ . \end{equation}
\item In the second case where $\lambda_2$ is a right factor of $\varphi^{(2)}_T=\lambda_1\circ \mu_1$, but not a right factor of $\mu_1$, the kernel of $\lambda_2\circ \lambda_1$ is annihilated by $T^2$ but not by $T$. So $\lambda_2\circ \lambda_1$ is a $T^2$-isogeny of rank $1$. As before, let $x_2$ be a nonzero $T$-torsion point in the kernel of $\lambda_2$.
Then, since $\varphi^{(2)}_T=\lambda_1\circ \mu_1$, the quantity
$$\mu_1(x_2)=-x_2^{q^2}-\frac{1}{u_1^{q+1}}x_2^q-\frac{1}{u_1}x_2$$
is a nonzero element of the kernel of $\lambda_1=\tau-u_1$, i.e., a root of $T^{q-1}-u_1$. So we have
\begin{eqnarray*}
&&\Bigl(-x_2^{q^2}-\frac{1}{u_1^{q+1}}x_2^q-\frac{1}{u_1}x_2\Bigr)^{q-1}-u_1\\
&&=x_2^{q-1}\cdot \Bigl(-(x_2^{q-1})^{q+1}-\frac{1}{u_1^{q+1}}x_2^{q-1}-\frac{1}{u_1} \Bigr)^{q-1}-u_1\\
&&=u_2\cdot \Bigl(-u_2^{q+1}-\frac{1}{u_1^{q+1}}u_2-\frac{1}{u_1} \Bigr)^{q-1}-u_1=0 \ . \\
\end{eqnarray*}
After clearing denominators, we obtain
\begin{equation}\label{eq:factor2}u_2\cdot(u_1^{q+1}u_2^{q+1}+u_2+u_1^q)^{q-1}-u_1^{q^2}=0 \ .\end{equation}
\end{itemize}

This behaviour also occurs when working with isomorphism classes. From Lemma \ref{lem:ghu} we see that the quantities $z_1:=u_1^{q^2+q+1}$ and $z_2:=u_2^{q^2+q+1}$ can be used to describe isomorphism classes. By raising both sides in Equation~\eqref{eq:u1u2} to the $(q^2+q+1)$-st power, we see that
\begin{equation}\label{eq:z1z2}\dfrac{(z_2-1)^{q^2+q+1}}{z_2^{q+1}}=g_2^{q^2+q+1}=\dfrac{(z_1-1)^{q^2+q+1}}{z_1^{q^2+q}} \ , \end{equation}
or equivalently
\begin{align*}
z_2^{q+1}(z_1-1)^{q^2+q+1}-(z_2-1)^{q^2+q+1}z_1^{q^2+q}=0 \ .
\end{align*}
Note that
\begin{align*}
z_2^{q+1}(z_1-1)^{q^2+q+1}-(z_2-1)^{q^2+q+1}z_1^{q^2+q}= F_1\cdot F_2 \ ,
\end{align*}
where
\begin{equation}\label{eq:factor1z}F_1=z_2(z_1-1)^{q+1}+(z_1-1)z_1^q(z_2-1)^q+z_1^{q+1}(z_2-1)^{q+1}
\quad \text{and}
\end{equation}
\begin{equation}\label{eq:factor2z} F_2=(z_1-1)z_2\cdot \Bigl(z_2(z_1-1)^{q+1}+(z_1-1)z_1^q(z_2-1)^q+z_1^{q+1}(z_2-1)^{q+1} \Bigr)^{q-1}-z_1^{q^2}(z_2-1)^{q^2} \ .
\end{equation}
These are the analogues of the factors described in Equations \eqref{eq:factor1} and \eqref{eq:factor2}.

\subsection{Composition of three $T$-isogenies}

The factor of degree $q+1$ found in Equation~\eqref{eq:factor1} corresponded to the situation of two $T$-isogenies $\lambda_1: \varphi^{(1)} \to \varphi^{(2)}$ and $\lambda_2: \varphi^{(2)} \to \varphi^{(3)}$ whose composition $\lambda_2 \circ \lambda_1$ is a $T$-isogeny of rank $2$. Now we consider a third $T$-isogeny $\lambda_3: \varphi^{(3)} \to \varphi^{(4)}$ and assume that $\lambda_3 \circ \lambda_2$ is a $T$-isogeny of rank $2$ as well. Writing $\lambda_3=\tau-u_3$, we see that
$$u_1^{q+1}u_2^{q+1}+u_2+u_1^q=0 \ \quad  \makebox{and}  \quad \ u_2^{q+1}u_3^{q+1}+u_3+u_2^q=0 \ .$$ However,
\begin{equation*}
0=u_2^{q+1}u_3^{q+1}+u_3+u_2^q=\left(u_2u_3-\frac{1}{u_1}\right)\left(u_2u_3\left(u_2u_3-\frac1{u_1}\right)^{q-1}-u_1u_2^q\right) \ .
\end{equation*}
These factors can be explained and obtained using modular theory (see Equation~\eqref{eq:factordeg1} and \eqref{eq:factordegq}).
\begin{itemize}
\item If $\lambda_3\circ \lambda_2 \circ \lambda_1=(\tau-u_3)(\tau-u_2)(\tau-u_1)$ is a $T$-isogeny of rank $3$, then $\varphi_T^{(1)}=-(\tau-u_3)(\tau-u_2)(\tau-u_1)$, implying that $u_3u_2u_1=1$, or equivalently
\begin{equation}\label{eq:factordeg1}u_2u_3-\frac{1}{u_1}=0 \ . \end{equation}
\item Assume that $\lambda_3\circ \lambda_2 \circ \lambda_1$ is not a $T$-isogeny. Since $\lambda_2\circ \lambda_1: \varphi^{(1)} \to \varphi^{(3)}$ (resp. $\lambda_3\circ \lambda_2: \varphi^{(2)} \to \varphi^{(4)}$) is a $T$-isogeny, it is a right factor of $\varphi_T^{(1)}$ (resp. $\varphi_T^{(2)}$). This implies that
    \begin{equation}\label{eq:tisogrk2}
    \varphi_T^{(1)}=\left(-\tau+\frac{1}{u_1u_2}\right)(\tau-u_2)(\tau-u_1) \ \makebox{and} \ \varphi_T^{(2)}=\left(-\tau+\frac{1}{u_2u_3}\right)(\tau-u_3)(\tau-u_2) \ .
    \end{equation}
    Since furthermore $(\tau-u_1)\varphi_T^{(1)}=\varphi_T^{(2)}(\tau-u_1),$ we see after canceling common factors that \begin{equation}\label{eq:threesteps}
    (\tau-u_1)\left(-\tau+\frac{1}{u_1u_2}\right)=\left(-\tau+\frac{1}{u_2u_3}\right)(\tau-u_3)\ .\end{equation}
    Denote by $x_3$ a nonzero element in $\ker \lambda_3$, so that $x_3^{q-1}=u_3$. Then $x_3$ is in the kernel of the righthand side in Equation~\eqref{eq:threesteps}. However, $x_3$ is not in the kernel of $\left(-\tau+1/(u_1u_2)\right),$ since this would give $\lambda_3=\left(\tau-1/(u_1u_2)\right)$ and hence that $\lambda_3\circ \lambda_2 \circ \lambda_1 =-\varphi_T^{(1)}$ would be a $T$-isogeny. Therefore, using Equation~\eqref{eq:threesteps}, we see that $-x_3^q+x_3/(u_1u_2)$ is a nonzero element of $\ker \lambda_1$, which implies that $$0=\left(-x_3^q+\frac{x_3}{u_1u_2}\right)^{q-1}-u_1=x_3^{q-1}\left(-x_3^{q-1}+\frac{1}{u_1u_2}\right)^{q-1}-u_1=u_3\left(-u_3+\frac{1}{u_1u_2}\right)^{q-1}-u_1 \ ,$$ or alternatively
    \begin{equation}\label{eq:factordegq}u_2u_3\left(u_2u_3-\frac{1}{u_1}\right)^{q-1}-u_1u_2^q=0 \ .\end{equation}
\end{itemize}
When passing to isomorphism classes, similar phenomena occur. First of all Equation \eqref{eq:factor1} is replaced by Equation \eqref{eq:factor1z}. Given that
$$z_2(z_1-1)^{q+1}+(z_1-1)z_1^q(z_2-1)^q+z_1^{q+1}(z_2-1)^{q+1}=0 \ ,$$
one then verifies that
\begin{multline*}
0=z_3(z_2-1)^{q+1}+(z_2-1)z_2^q(z_3-1)^q+z_2^{q+1}(z_3-1)^{q+1}\\
=\left(z_2z_3-\frac{1}{z_1}\right)\left((z_2z_3-1)\left(z_2z_3-\frac{1}{z_1}\right)^{q-1}-\frac{(z_2-1)^q}{z_2}-\frac{(z_1-1)^q}{z_1^q}\right)\ .
\end{multline*}
The analogues of Equations \eqref{eq:factordeg1} and \eqref{eq:factordegq} when working with isomorphism classes are therefore given by
\begin{equation*}
z_2z_3-\frac{1}{z_1}=0
\end{equation*}
and
\begin{equation}\label{eq:factordegqz}(z_2z_3-1)\left(z_2z_3-\frac{1}{z_1}\right)^{q-1}-\frac{(z_2-1)^q}{z_2}-\frac{(z_1-1)^q}{z_1^q}=0\ .\end{equation}

\subsection{Composition of more than three isogenies}

We finish this section by giving some information on the composite of more than three $T$-isogenies.
As before let $\phi^{(i)}$ be Drinfeld modules given by $\phi_{T}^{(i)}=-\tau^{3}+g_{i}\tau+1$ and let $\lambda_i:\phi^{(i)}\mapsto\phi^{(i+1)}$ be $T$-isogenies of rank $1$ with $\lambda_i=\tau-u_i$ for nonzero element $u_i\in L$. In this subsection we always assume that the composite of two $T$-isogenies $\lambda_i\circ\lambda_{i+1}:\phi^{(i)}\mapsto\phi^{(i+2)} $ is a $T$-isogeny of rank $2$, but that the composite of three $T$-isogenies $\lambda_i\circ\lambda_{i+1}\circ \lambda_{i+2}:\phi^{(i)}\mapsto\phi^{(i+3)} $ is not a $T$-isogeny. These assumptions correspond to the case in the previous subsections leading to Equations \eqref{eq:factor1z} and \eqref{eq:factordegqz}.
Using this we show the following theorem.
\begin{theorem}
Assume that for any $i \ge 1$, the composite of two $T$-isogenies $\lambda_i\circ\lambda_{i+1}$ is a $T$-isogeny of rank $2$, but that the composite of three $T$-isogenies $\lambda_{i+2}\circ\lambda_{i+1}\circ \lambda_{i}$ is not a $T$-isogeny.
Then the composite of $2k$ many such $T$-isogenies $\lambda_{i+2k-1}\circ \cdots \circ \lambda_i$ is a $T^k$-isogeny of rank $2$.
\end{theorem}
\begin{proof}
First note that similarly as in Equation \eqref{eq:tisogrk2}, the assumption that the composite of two $T$-isogenies $\lambda_i\circ\lambda_{i+1}$ is a $T$-isogeny of rank $2$, implies that:
\begin{align}\label{eq:etalambda}
\phi_{T}^{(i)}=\eta_i\circ\lambda_{i+1}\circ\lambda_i \ , \quad
\phi_{T}^{(i+1)}=\lambda_i\circ\eta_i\circ\lambda_{i+1} \quad \text{and} \quad
\phi_{T}^{(i+2)}=\lambda_{i+1}\circ\lambda_i\circ\eta_i  \ ,
\end{align}
where $\eta_i=-(\tau-1/(u_{i+1}u_i))$.
We firstly prove that for any integers $i,k\geq 1$ we have
\begin{align*}
\phi_{T^{k}}^{(i)}=\eta_i\circ \eta_{i+2}\circ\cdots \circ\eta_{i+2k-2}\circ\lambda_{i+2k-1}\circ\lambda_{i+2k-2}\circ\cdots \circ\lambda_{i+1}\circ\lambda_i \ .
\end{align*}
The proof is induction on $k$. By our assumption, the argument is trivially true for $k=1$. Now suppose that it is true for $k=n$. Then by using Equation \eqref{eq:etalambda} we obtain the following equalities, which concludes the desired argument.
$$\begin{array}{lcl}
\phi_{T^{n+1}}^{(i)} & = & \phi_{T^{n}}^{(i)} \circ \phi_{T}^{(i)} \\
 & = & \left(\eta_i\circ \eta_{i+2}\circ\cdots \circ\eta_{i+2n-2}\circ\lambda_{i+2n-1}\circ\lambda_{i+2n-2}\circ\cdots \circ\lambda_{i+1}\circ\lambda_i\right) \circ \left( \eta_i\circ\lambda_{i+1}\circ\lambda_i \right)\\
& = & \left(\eta_i\circ \eta_{i+2}\circ\cdots \circ\eta_{i+2n-2}\circ\lambda_{i+2n-1}\circ\lambda_{i+2n-2}\circ\cdots \circ\lambda_{i+1} \right)\circ\left( \lambda_i \circ  \eta_i \circ\lambda_{i+1} \right)\circ\lambda_i \\
& = &  \left(\eta_i\circ \eta_{i+2}\circ\cdots \circ\eta_{i+2n-2}\circ\lambda_{i+2n-1}\circ\lambda_{i+2n-2}\circ\cdots \circ\lambda_{i+1} \right)\circ \phi_T^{(i+1)} \circ\lambda_i\\
& = &  \left(\eta_i\circ \eta_{i+2}\circ\cdots \circ\eta_{i+2n-2}\circ\lambda_{i+2n-1}\circ\lambda_{i+2n-2}\circ\cdots \circ\lambda_{i+1} \right)\circ\left( \eta_{i+1} \circ \lambda_{i+2} \circ \lambda_{i+1}\right)\circ\lambda_i\\
 & \vdots & \\
 & = & \eta_i\circ \eta_{i+2}\circ\cdots \circ\eta_{i+2n}\circ\lambda_{i+2n+1}\circ\lambda_{i+2n}\circ\cdots \circ\lambda_{i+1}\circ\lambda_i \, .
\end{array}
$$
Now set $\mathbf{K}_{n}:=\mathrm{ker}(\lambda_{i+2n-1}\circ \cdots \circ \lambda_i)$ for $n\geq 1$. From above equalities we see that $\mathbf{K}_{n+1}$ is annihilated by $\phi_{T^{n+1}}^{(i)}$, and hence $\mathbf{K}_{n+1}$ is an $\mathbb{F}[T]/\langle T^{n+1} \rangle$ module of cardinality $q^{2n+2}$. We consider the map $m_T:\mathbf{K}_{n+1}\to \mathbf{K}_{n}$ defined by $m_T(a):=\phi_{T}^{(i)}(a)$ for $a\in \mathbf{K}_{n+1}$. Note that since $\mathbf{K}_{n+1}$ is annihilated by $\phi_{T^{n+1}}^{(i)}$, the map $m_T$ is well-defined homomorphism. It is clear that $\mathrm{ker}(\lambda_{i+1}\circ \lambda_i)$ lies in the set $\mathrm{ker}(m_T)$, and hence the cardinality of $\mathrm{ker}(m_T)$ is at least $q^2$. On the other hand, $\mathrm{ker}(m_T)$ contain at most $q^3$ elements since $\mathrm{ker}(m_T)$ lies in $\phi^{(i)}[T]$.

Now we prove that the cardinality of $\mathrm{ker}(m_T)$ is equal to $q^2$. Suppose our claim is not true. Then $\phi^{(i)}[T]\subseteq \mathbf{K}_{n+1}$ and this implies that
\begin{align*}
\lambda_{i+2n+1}\circ \cdots \circ \lambda_i=\psi_i\circ \phi_{T}^{(i)}=\psi_i\circ \eta_i\circ\lambda_{i+1}\circ\lambda_i \ ,
\end{align*}
for some $\psi_i\in L\lbrace\tau\rbrace$. In other words, $\lambda_{i+2n+1}\circ \cdots \circ \lambda_{i+2}=\psi_i\circ \eta_i$, where as before $\eta_i=-\left(\tau-1/(u_{i+1}u_i) \right)$. Let $1/(x_{i+1}x_i)$ be a nonzero torsion point of $\eta_i$. Since by assumption $\lambda_{i+2}\circ\lambda_{i+1}\circ \lambda_{i}$ is not a $T$-isogeny, there exists an integer $j$ with $2<j\leq 2n+1$ such that $1/(x_{i+1}x_i)$ is a nonzero torsion point of $\lambda_{i+j}\circ\cdots \circ \lambda_{i+2}$ but it is not a torsion point of $\lambda_{i+j-1}\circ\cdots \circ \lambda_{i+2}$. Hence the polynomial $f(T)=(T^q-u_{i+j-1}T)\circ \cdots \circ (T^q-u_{i+2}T)$ evaluated at $1/(x_{i+1}x_i)$ is a nonzero torsion point of $\lambda_{i+j}$.
This means that $\left( f\left(1/(x_{i+1}x_i) \right) \right)^{q-1}=u_{i+j}$. Note that
\begin{align*}
f(T)= T^{q^{j-2}}+a_{j-3} T^{q^{j-3}}+\cdots+ a_1T^q+a_0T \, ,
\end{align*}
where the $a_\ell$ are polynomials in $\mathbb{Z}[u_{i+j-1}, \ldots , u_{i+2}]$. As a result,
\begin{align}\label{uj}
& u_{i+j}=f\left(\frac{1}{x_{i+1}x_i} \right)^{q-1}\\
& =\left(\frac{1}{x_{i+1}x_i} \right)^{q-1}\left(\left(\frac{1}{x_{i+1}x_i} \right)^{q^{j-2}-1}+a_{j-3} \left(\frac{1}{x_{i+1}x_i} \right)^{q^{j-3}-1}+\cdots+ a_1\left(\frac{1}{x_{i+1}x_i} \right)^{q-1}+a_0 \right)^{q-1} \, . \nonumber
\end{align}
Since $1/(x_{i+1}x_i)$ is a nonzero torsion point of $\eta_i$, we have $\left( 1/(x_{i+1}x_i)\right)^{q-1}=1/(u_{i+1}u_i)$. By Equation \eqref{uj}, we can express $u_{i+j}$ as a rational function of $u_{i+j-1},\ldots, u_{i+2},u_{i+1},u_i$ with coefficients in the prime field. Then we have the following figure:
\begin{figure}[h]
\begin{small}
\begin{displaymath}
\xymatrix{
\mathbb{F}(u_1,\ldots,u_{i+j-1}) \ar@{-}[rrr]_{\mathrm{deg}=1} &  & & \mathbb{F}(u_1,\ldots,u_{i+j}) & \\
\mathbb{F}(z_1,\ldots,z_{i+j-1}) \ar@{-}@[red][u]_{\quad \quad \Longrightarrow \quad \quad \text{\textcolor{red}{tame extensions}}\quad \quad \Longleftarrow \quad \quad} \ar@{-}[rrr]_{\mathrm{deg}=q}&  & & \mathbb{F}(z_1,\ldots, z_{i+j}) \ar@{-}@[red][u]  &
}
\end{displaymath}
\end{small}
\label{fig:ZG}
\end{figure}
Since $z_i=u_i^{q^2+q+1}$ for all values of $i$, the extensions $\fqc(u_1,\dots,u_{i+j})/\fqc(z_1,\dots,z_{i+j})$ and $\fqc(u_1,\dots,u_{i+j-1})/\fqc(z_1,\dots,z_{i+j-1})$ are tame. Moreover, as shown in \cite[Sec.2.1]{abn2016}, the extension degree of $\fqc(z_1,\dots,z_{i+j})/\fqc(z_1,\dots,z_{i+j-1})$, equals $q$. Therefore we obtain a contradiction. Hence $\mathrm{ker}(m_T)$ has $q^2$ elements.

Combining the structure theorem for finitely generated modules over a principal ideal domain with the fact that $\mathrm{ker}(m_T)$ has $q^2$ elements, we may conclude that $\mathbf{K}_{n+1}$ is isomorphic to a direct sum of exactly two cyclic submodules, i.e.: $$\mathbf{K}_{n+1} \cong \langle T^{\ell_1} \rangle /\langle T^{n+1}\rangle \bigoplus \langle T^{\ell_2} \rangle /\langle T^{n+1}\rangle \, ,$$
for some integers $\ell_1, \ell_2\geq 0$.
Further, since the cardinality of $\mathbf{K}_{n+1}$ is $q^{2n+2}$, we obtain that $\ell_1=\ell_2= 0$. In other words, $\lambda_{i+2n+1}\circ \cdots \circ \lambda_i$ is a $T^{n+1}$-isogeny of rank $2$.
\end{proof}

\section{Two recursive cubic towers of function fields}

A sequence of function fields $\mathcal F = (F_1 \subseteq F_2 \subseteq \ldots )$ over $\fq$ is called recursive if there exists $f(X,Y)\in \mathbb{F}_q[X,Y]$ such that
\begin{itemize}
\item [(i)] there exists $x_1 \in F_1$ transcendental over $\fq$, such that $F_1=\fq(x_1)$, and
\item [(ii)] $F_{i+1}=F_i(x_{i+1})$ for some $x_{i+1}$ satisfying $f(x_i,x_{i+1})=0$, for all $i\geq 1$.
\end{itemize}
The polynomial $f(X,Y)$ may not determine the sequence $\mathcal F$ uniquely, since it may happen that the polynomial $f(x_i,T) \in F_i[T]$ is reducible for some $i$. In such cases, there may exist distinct sequences of function fields satisfying the same recursion. In other words, the polynomial $f(X,Y)$ may give rise to several recursive sequences of function fields as given in Remark \ref{towers}.

A recursive sequence of function fields $\mathcal F$ is called a recursive tower with constant field $\mathbb{F}_q$, if
\begin{itemize}
\item [(i)] for all $i \ge 1$, the finite field $\mathbb{F}_q$ is the full constant field of $F_i$, and
\item [(ii)] $\lim_{i \to \infty} g(F_{i})=\infty$.
\end{itemize}
Recursive towers of function fields have been used to achieve good lower bounds on Ihara's constant $A(q)$, see \cite{BaBeGS2013}. Cubic, recursive towers (whose field of definition we will denote by $\mathbb{F}_{q^3}$) can be obtained from the modular setting in a natural way. For $i \ge 1$, let $\varphi^{(i)}$ be normalized weakly supersingular Drinfeld modules given by $\varphi_T^{(i)}=-\tau^3+g_i\tau^2+1$ and let $\lambda^{(i)}: \varphi^{(i)} \to \varphi^{(i+1)}$ be $T$-isogenies of the form $\lambda^{(i)}=\tau-u_i$. In the previous section we have seen that the variables $g_i$ and $u_j$ are related to each other in an algebraic way. These relations give rise to several recursive towers of function fields. As Lemma \ref{lem:ghu} suggests, it is natural to pass to the variables $z_i:=u_i^{q^2+q+1}$. We then find two cubic, recursive sequences of function fields:

We denote by $\mathcal D=( D_1 \subset D_2 \subset\cdots)$ a recursive sequence of function fields satisfying $D_1=\mathbb{F}_{q^3}(z_1)$ and $D_{i+1}=D_i(z_{i+1})$, with $$z_2(z_1-1)^{q+1}+(z_1-1)z_1^q(z_2-1)^q+z_1^{q+1}(z_2-1)^{q+1}=0$$ and  $$(z_iz_{i+1}-1)\left(z_iz_{i+1}-\frac{1}{z_{i-1}}\right)^{q-1}-\frac{(z_i-1)^q}{z_i}-\frac{(z_{i-1}-1)^q}{z_{i-1}^q}=0 \ \makebox{for $i \ge 2$}\ .$$

Further we denote by $\mathcal E=(E_1 \subset E_2 \subset \cdots)$ a recursive sequence of function fields satisfying $E_1=\mathbb{F}_{q^3}(z_1)$ and $E_{i+1}=E_i(z_{i+1})$, with $$(z_i-1)z_{i+1}\cdot \Bigl(z_{i+1}(z_i-1)^{q+1}+(z_i-1)z_i^q(z_{i+1}-1)^q+z_i^{q+1}(z_{i+1}-1)^{q+1} \Bigr)^{q-1}-z_i^{q^2}(z_{i+1}-1)^{q^2}=0$$ for $i \ge 1$.

\begin{remark}\label{towers}
By definition, the sequence $\mathcal E$ satisfies the recursion
$$(X-1)Y\cdot \Bigl(Y(X-1)^{q+1}+(X-1)X^q(Y-1)^q+X^{q+1}(Y-1)^{q+1} \Bigr)^{q-1}-X^{q^2}(Y-1)^{q^2}=0 \ .$$
Comparing the definition of sequence $\mathcal D$ with Equations \eqref{eq:factor1z} and \eqref{eq:factordegqz}, we see that sequence $\mathcal D$ satisfies the recursion
$$Y(X-1)^{q+1}+(X-1)X^q(z_2-1)^q+X^{q+1}(Y-1)^{q+1}=0.$$
Since Equations \eqref{eq:factor1z} and \eqref{eq:factor2z} arise as factors of the same function given in Equation \eqref{eq:z1z2}, we conclude that both sequences $\mathcal D$ and $\mathcal E$ satisfy the same recursion, namely
$$Y^{q+1}(X-1)^{q^2+q+1}-(Y-1)^{q^2+q+1}X^{q^2+q}=0\ .$$
\end{remark}

\begin{remark}\label{rem:towerD}
In \cite{abn2016} (see the proof of Lemma 2.6 there) it is shown that the sequence $\mathcal D$ is in fact a tower with full constant field $\mathbb{F}_{q^3}$. More precisely, it is shown there that for the tower $\mathcal D$ it holds that $[D_2:D_1]=q+1$ and $[D_{i+1}:D_i]=q$ if $i>1$.
Moreover, it is shown that
\begin{align*}
z_1=(\alpha_1+1)/\alpha_1^{q+1} \qquad \text{and} \qquad z_2=\alpha_1^{q+1}+\alpha_1 \ ,
\end{align*}
where $\alpha_1:=(z_1z_2-1)/(z_1+1)$. This shows in particular that the second function field $D_2$ is rational. If for $i \ge 1$ we define $C_i:=D_{i+1}$ and $\alpha_i:=(z_iz_{i+1}-1)/(z_i+1)$, we obtain a tower $\mathcal C=( C_1 \subset C_2 \subset\cdots)$ with variables $\alpha_1,\alpha_2,\dots$ satisfying the recursion
\begin{align*}
 \frac{\alpha_{i+1}+1}{\alpha_{i+1}^{q+1}}=\alpha_i^{q+1}+\alpha_i\, \ .
\end{align*}
\end{remark}

\begin{remark}
Tower $\mathcal E$ is in fact (up to a change of variables) the same as the particular case $n=3$, $j=2$ and $k=1$ of the tower $\mathcal H$ studied in \cite{BaBeGS2013}. It turns out that $\mathcal E$ can be seen as a subtower of $\mathcal D$, see \cite{abn2016} for details. This relation can be depicted as in Figure \ref{fig:CDE}.
\begin{tiny}
\begin{figure}[h]
\begin{displaymath}
\xymatrix{
D_1 \ar@{-}@[red][r]^{\mathrm{deg}=q+1}& D_2 \ar@{-}@[red][r] & D_3 \ar@{-}@[red][r] & D_4\ar@{.}@[red][rr] &&   D_{2i} \ar@{-}@[red][r]&D_{2i+1} \ar@{-}@[red][r] & D_{2i+2} \ar@{.}@[red][r]&\\
& C_1 \ar@{-}[u]^{=}  \ar@{-}@[red][r]^{\mathrm{deg}=q}& C_2  \ar@{-}[u]^{=} \ar@{-}@[red][r]^{\mathrm{deg}=q}& C_3 \ar@{.}@[red][rr] \ar@{-}[u]^{=} & & C_{2i-1}\ar@{-}[u]^{=} \ar@{-}@[red][r]^{\mathrm{deg}=q} & C_{2i}\ar@{-}[u]^{=} \ar@{-}@[red][r] ^{\mathrm{deg}=q}& C_{2i+1}\ar@{-}[u]^{=} \ar@{.}@[red][r] &\\
&&&&&&&
\\
& E_1 \ar@{-}@[red][uu]^{\mathrm{deg}=q+1} \ar@{-}@[red][rr]^{\mathrm{deg}=q^2} & & E_2 \ar@{-}@[red][uu]^{\mathrm{deg}=q+1} \ar@{.}@[red][rr]& &   E_i \ar@{-}@[red][uu]^{\mathrm{deg}=q+1} \ar@{-}@[red][rr]^{\mathrm{deg}=q^2} & &E_{i+1} \ar@{-}@[red][uu]^{\mathrm{deg}=q+1} \ar@{.}@[red][r]&
}
\end{displaymath}
\caption{Relationship between towers $\mathcal C$, $\mathcal D$ and $\mathcal E$}
\label{fig:CDE}
\end{figure}
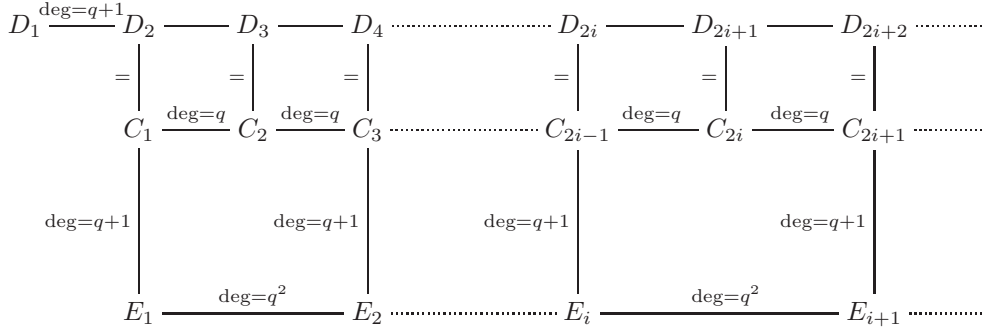
\end{tiny}
\end{remark}

\section{Relation to previously known cubic towers}

In this section we give an overview of several previously studied, cubic, recursive towers for which no modular interpretation was known. These towers all have limit at least $2(q^2-1)/(q+2)$, which is Zink's bound for $A(q^3)$. We will also see that they essentially all are equal to tower $\mathcal D$. Since the defining equations of tower $\mathcal D$ are explained by the theory of Drinfeld modules, we will then have reached our goal.

The first explicit tower known to achieve Zink's bound was given in \cite{vdGvdV2002}. This tower is defined recursively over the field $\f_8$ by the equation $$Y^2+Y=\dfrac{1}{X}+1+X.$$
In \cite{BGS2005} this tower was generalized for any $q$. The tower found in \cite{BGS2005} is recursively defined by the equation
$$\dfrac{Y^q}{1-Y}=\dfrac{X}{X^q+X-1}\ ,$$
or equivalently, after the change of variables $x=1/X$ and $y=1/Y$, by
\begin{equation}\label{eq:BGS}
y^q-y^{q-1}=\dfrac{1}{x^{q-1}}+1-x \ .
\end{equation}
We denote the tower recursively defined by Equation \eqref{eq:BGS} by ${\mathcal A=(A_1\subset A_2 \subset \cdots)}$ and the variables by $x_1,x_2,\dots$. Given in this form, it is clear that for $q=2$, the tower $\mathcal A$ reduces to the tower in \cite{vdGvdV2002}.

In \cite{I2007,CG2012} it was shown that the tower in \cite{BGS2005} also can be defined by the reducible equation
\begin{equation}\label{eq:iharared}
v^{q+1}+v=\dfrac{u+1}{u^{q+1}} \ .
\end{equation}
More precisely, the observations in \cite{I2007} imply that the quantities $U=x^q-x^{q-1}$ and $V=1/x^{q-1}+1-x$ satisfy the equation $$\dfrac{-V^q}{(1-V)^{q+1}}=\dfrac{U-1}{U^{q+1}}\ .$$ After the substitution of variables $u=-1+1/U$ and $v=-1+1/V$, Equation \eqref{eq:iharared} follows. Note that the polynomial $v^{q+1}+v-(u+1)/u^{q+1} \in \fqc(u)[v]$ is the product of $v+(u+1)/u$ and an irreducible factor of degree $q$ given by the left-hand-side of the equation:
\begin{equation}\label{eq:ihara}1+\sum_{i=0}^{q}v^i\left(-\frac{u+1}{u}\right)^{q-i}=0\ .\end{equation}
This equation is used to define a tower $\mathcal B = (B_1 \subset B_2 \subset \cdots)$ with constant field $\fqc$. Note that since $u$ and $v$ generate the rational function field $\fqc(x)$ (since $1/x=u(v+1)$), the towers $\mathcal A$ and $\mathcal B$ recursively defined by Equations \eqref{eq:BGS} and \eqref{eq:ihara} are essentially the same. To be precise, for $i \ge 1$ we have $A_i=B_{i+1}$ or in other words: if one deletes the first function field of the tower defined by Equation \eqref{eq:ihara}, one obtains the tower defined by \eqref{eq:BGS}. The towers therefore have the same limit. This limit was computed in \cite{BeGS2005} (using results from \cite{Be2004} and \cite{BGS2005}) to be $2(q^2-1)/(q+2)$.

Given a recursive tower $\mathcal F$ satisfying the recursion $F(x,y) = 0$, we define the dual tower of $\mathcal F$ to be the recursive tower satisfying the recursion $F(y,x)=0.$ Essentially, the order of the variables is interchanged. In particular, this means that reversing the order of the variables, gives an isomorphism between the $i$-th function fields of tower and its dual. Comparing Remark \ref{rem:towerD} and Equation \eqref{eq:iharared}, we conclude that the towers $\mathcal C$ and $\mathcal B$ are duals of each other. In particular that $C_i \cong B_i$ for all $i \ge 1$. In particular, the towers $\mathcal A$ and $\mathcal B$ can be obtained and explained using the theory of normalized Drinfeld modules of rank $3$.

\bigskip
\noindent
Concluding, an overview of the relation between towers $\mathcal A, \mathcal B, \mathcal C$ and  $\mathcal D$ is as in the following figure:
\begin{tiny}
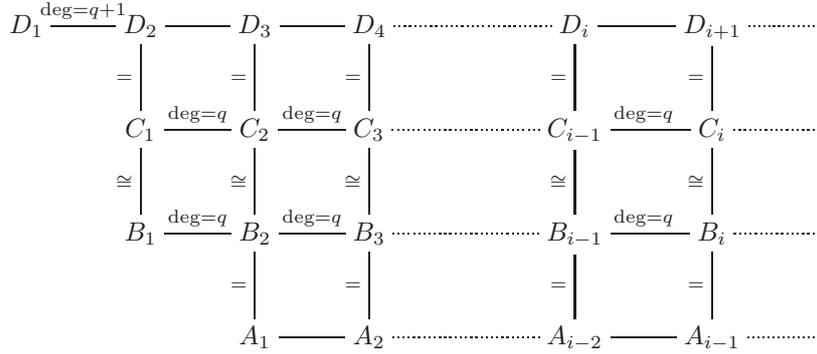
\begin{figure}[h]
\begin{displaymath}
\xymatrix{
D_1 \ar@{-}@[red][r]^{\mathrm{deg}=q+1}& D_2 \ar@{-}@[red][r] & D_3 \ar@{-}@[red][r] & D_4\ar@{.}@[red][rr] &&   D_{i} \ar@{-}@[red][r]&D_{i+1} \ar@{.}@[red][r] & \\
& C_1 \ar@{-}[u]^{=} \ar@{-}@[red][r]^{\mathrm{deg}=q}& C_2  \ar@{-}[u]^{=} \ar@{-}@[red][r]^{\mathrm{deg}=q}& C_3 \ar@{.}@[red][rr] \ar@{-}[u]^{=}& & C_{i-1}\ar@{-}[u]^{=}\ar@{-}@[red][r]^{\mathrm{deg}=q} & C_{i}\ar@{-}[u]^{=}\ar@{.}@[red][r] & \\
& B_1 \ar@{-}[u]^{\cong} \ar@{-}@[red][r]^{\mathrm{deg}=q} & B_2 \ar@{-}[u]^{\cong} \ar@{-}@[red][r]^{\mathrm{deg}=q}& B_3 \ar@{-}[u]^{\cong} \ar@{.}@[red][rr] & &B_{i-1} \ar@{-}[u]^{\cong} \ar@{-}@[red][r]^{\mathrm{deg}=q} & B_i \ar@{-}[u]^{\cong} \ar@{.}@[red][r] & \\
& & A_1 \ar@{-}[u]^{=} \ar@{-}@[red][r] & A_2\ar@{-}[u]^{=} \ar@{.}@[red][rr] & & A_{i-2} \ar@{-}[u]^{=} \ar@{-}@[red][r] & A_{i-1} \ar@{-}[u]^{=} \ar@{.}@[red][r] &
}
\end{displaymath}
\caption{Relationship between towers $\mathcal A$, $\mathcal B$, $\mathcal C$ and $\mathcal D$.}
\label{fig:ABCD}
\end{figure}
\end{tiny}

\section{Acknowledgements}

The authors would like to gratefully acknowledge the following foundations for received support. Nurdag\"ul Anbar is supported by H.C.~\O rsted COFUND Post-doc for the project ``Algebraic curves with many rational points" and by The Danish Council for Independent Research (Grant No.~DFF--4002-00367). Alp Bassa is supported by T\"{u}bitak Proj.~No.~112T233 and by the BAGEP Award of the Science Academy with funding supplied by Mehve{\c s} Demiren in memory of Selim Demiren. Peter Beelen is supported by The Danish Council for Independent Research (Grant No.~DFF--4002-00367).

\bibliographystyle{99}

\noindent
Nurdag\"ul Anbar\\
Technical University of Denmark,
Department of Applied Mathematics and Computer Science,
Matematiktorvet 303B, 2800 Kgs. Lyngby,
Denmark,
nurdagulanbar2@gmail.com

\vspace{1ex}
\noindent
Alp Bassa\\
Bo\u{g}azi\c{c}i University,
Faculty of Arts and Sciences,
Department of Mathematics,
34342 Bebek, \.{I}stanbul,
Turkey,
alp.bassa@boun.edu.tr

\vspace{1ex}
\noindent
Peter Beelen\\
Technical University of Denmark,
Department of Applied Mathematics and Computer Science,
Matematiktorvet 303B, 2800 Kgs. Lyngby,
Denmark,
pabe@dtu.dk

\end{document}